\documentclass[12pt]{amsart}
\RequirePackage{mathtools}
\RequirePackage{amsopn}
\RequirePackage{amsfonts}
\RequirePackage{paralist}
\RequirePackage{amssymb}
\RequirePackage{amsthm}
\RequirePackage{mathrsfs}
\usepackage[alphabetic]{amsrefs}
\renewcommand\MR[1]{\relax} 
\usepackage{tikz}
\usetikzlibrary{cd,decorations.pathmorphing,decorations.markings}
\mathtoolsset{showonlyrefs}
\newtheorem{thm}{Theorem}[section] 
\numberwithin{equation}{section}

\newtheorem{cor}[thm]{Corollary}
\newtheorem{lemma}[thm]{Lemma}
\newtheorem{prop}[thm]{Proposition}

\theoremstyle{definition}
\newtheorem{definition}[thm]{Definition}

\theoremstyle{remark}
\newtheorem{remark}[thm]{Remark}
\newtheorem{example}[thm]{Example}

\def\mathcs{C^{*}}
\newcommand{\cs}{\ensuremath{\mathcs}}

\DeclareMathSymbol{\rtimes}{\mathbin}{AMSb}{"6F}
\newcommand{\ib}{im\-prim\-i\-tiv\-ity bi\-mod\-u\-le}

\newcommand{\sme}{\,\mathord{\mathop{\text{--}}\nolimits_{\relax}}\,}
\def\ibind#1{\mathop{#1\mathord{\mathop{\text{--}}}}\!\Ind\nolimits}
\newcommand\xind{\ibind\X}
\newcommand\R{\mathbf{R}}
\newcommand\C{\mathbf{C}}

\newcommand\set[1]{\{\,#1\,\}}
\newcommand\sset[1]{\{#1\}}
\let\tensor=\otimes
\def\restr#1{|_{{#1}}}
\makeatletter
\def\labelenumi{\textnormal{(\@alph\c@enumi)}}
\def\theenumi{\@alph \c@enumi}
\def\labelenumii{\textnormal{(\@roman\c@enumii)}}
\def\theenumii{\@roman \c@enumii}
\newcount\charno
\def\alphapart#1{\charno=96
\advance\charno by#1\char\charno}

\makeatother

\def\<{\langle}
\def\>{\rangle}
\let\ipscriptstyle=\scriptscriptstyle
\def\lipsqueeze{{\mskip -3.0mu}}
\def\ripsqueeze{{\mskip -3.0mu}}
\def\ipcomma{\nobreak\mathrel{,}\nobreak}
\newbox\ipstrutbox
\setbox\ipstrutbox=\hbox{\vrule height8.5pt
width 0pt}
\def\ipstrut{\copy\ipstrutbox}
\def\lip#1<#2,#3>{\mathopen{\relax_{\ipstrut\ipscriptstyle{
#1}}\lipsqueeze
\langle} #2\ipcomma #3 \rangle}
\def\blip#1<#2,#3>{\mathopen{\relax_{\ipstrut
\ipscriptstyle{ #1}}\lipsqueeze\bigl\langle} #2\ipcomma #3 \bigr\rangle}
\def\rip#1<#2,#3>{\langle #2\ipcomma #3
\rangle_{\ripsqueeze\ipstrut\ipscriptstyle{#1}}}
\def\brip#1<#2,#3>{\bigl\langle #2\ipcomma #3
\bigr\rangle_{\ripsqueeze\ipstrut\ipscriptstyle{#1}}}
\def\angsqueeze{\mskip -6mu}
\def\smangsqueeze{\mskip -3.7mu}
\def\trip#1<#2,#3>{\langle\smangsqueeze\langle #2\ipcomma #3
\rangle\smangsqueeze\rangle_{\ripsqueeze\ipstrut\ipscriptstyle{#1}}}
\def\btrip#1<#2,#3>{\bigl\langle\angsqueeze\bigl\langle #2\ipcomma
#3
\bigr\rangle
\angsqueeze\bigr\rangle_{\ripsqueeze\ipstrut\ipscriptstyle{#1}}}
\def\tlip#1<#2,#3>{\mathopen{\relax_{\ipstrut\ipscriptstyle{
#1}}\lipsqueeze \langle\smangsqueeze\langle} #2\ipcomma #3
\rangle\smangsqueeze\rangle}
\def\btlip#1<#2,#3>{\mathopen{\relax_{\ipstrut\ipscriptstyle{
#1}}\lipsqueeze
\bigl\langle\angsqueeze\bigl\langle} #2\ipcomma #3
\bigr\rangle\angsqueeze\bigr\rangle}

\def\ip(#1|#2){(#1\mid #2)}
\def\bip(#1|#2){\bigl(#1 \mid #2\bigr)}
\def\Bip(#1|#2){\Bigl( #1 \bigm| #2 \Bigr)}
\expandafter\ifx\csname BibSpec\endcsname\relax\else
\BibSpec{collection.article}{%
    +{}  {\PrintAuthors}                {author}
    +{,} { \textit}                     {title}
    +{.} { }                            {part}
    +{:} { \textit}                     {subtitle}
    +{,} { \PrintContributions}         {contribution}
    +{,} { \PrintConference}            {conference}
    +{}  {\PrintBook}                   {book}
    +{,} { }                            {booktitle}
    +{,} { }                            {series}
    +{,} { \voltext}                    {volume}
    +{,} { }                            {publisher}
    +{,} { }                            {organization}
    +{,} { }                            {address}
    +{,} { \PrintDateB}                 {date}
    +{,} { pp.~}                        {pages}
    +{,} { }                            {status}
    +{,} { \PrintDOI}                   {doi}
    +{,} { available at \eprint}        {eprint}
    +{}  { \parenthesize}               {language}
    +{}  { \PrintTranslation}           {translation}
    +{;} { \PrintReprint}               {reprint}
    +{.} { }                            {note}
    +{.} {}                             {transition}
}
\BibSpec{article}{%
    +{}  {\PrintAuthors}                {author}
    +{,} { \textit}                     {title}
    +{.} { }                            {part}
    +{:} { \textit}                     {subtitle}
    +{,} { \PrintContributions}         {contribution}
    +{.} { \PrintPartials}              {partial}
    +{,} { }                            {journal}
    +{}  { \textbf}                     {volume}
    +{}  { \PrintDatePV}                {date}
    +{,} { \eprintpages}                {pages}
    +{,} { }                            {status}
    +{,} { \PrintDOI}                   {doi}
    +{,} { available at \eprint}        {eprint}
    +{}  { \parenthesize}               {language}
    +{}  { \PrintTranslation}           {translation}
    +{;} { \PrintReprint}               {reprint}
    +{.} { }                            {note}
    +{.} {}                             {transition}
}
\BibSpec{book}{%
    +{}  {\PrintPrimary}                {transition}
    +{,} { \textit}                     {title}
    +{.} { }                            {part}
    +{:} { \textit}                     {subtitle}
    +{,} { \PrintEdition}               {edition}
    +{}  { \PrintEditorsB}              {editor}
    +{,} { \PrintTranslatorsC}          {translator}
    +{,} { \PrintContributions}         {contribution}
    +{,} { }                            {series}
    +{,} { \voltext}                    {volume}
    +{,} { }                            {publisher}
    +{,} { }                            {organization}
    +{,} { }                            {address}
    +{,} { pp.~}                        {pages}
    +{,} { \PrintDateB}                 {date}
    +{,} { }                            {status}
    +{}  { \parenthesize}               {language}
    +{}  { \PrintTranslation}           {translation}
    +{;} { \PrintReprint}               {reprint}
    +{.} { }                            {note}
    +{.} {}                             {transition}
}
\fi
\evensidemargin=\oddsidemargin
\newcommand\X{\mathsf{X}}
\newcommand\Y{\mathsf{Y}}
\newcommand\Span{\operatorname{span}}
\newcommand\ospan{\overline{\Span}}
\newcommand\I{\mathcal{I}}

\newcommand\Ind{\operatorname{Ind}}

\newcommand\grg{\mathcal{G}}

\newcommand\Prim{\operatorname{Prim}}

\newcommand\B{\mathscr{B}}
\newcommand\E{\mathscr{E}}
\newcommand\CC{\mathscr{C}}
\newcommand\JJ{\mathscr{J}}
\newcommand\KK{\mathscr{K}}
\newcommand\A{\mathscr{A}}
\newcommand\pb{p_{\B}}

\newcommand\pc{p_{\CC}}
\newcommand\qe{q_{\E}}
\newcommand\M{\mathscr{M}}
\newcommand\JJM{\JJ_{\M}}
\newcommand\KKM{\KK^{\M}}
\newcommand\qk{q^{\KK}}
\newcommand\qj{q^{\JJ}}

\newcommand\go{G^{(0)}}
\newcommand\ho{H^{(0)}}
\newcommand\ko{K^{(0)}}

\newcommand\Ex{\operatorname{Ex}}
\newcommand\Lip{\tlip\scriptstyle *}
\newcommand\Rip{\trip\scriptstyle *}

\usepackage[colorinlistoftodos,textsize=tiny,textwidth=1in]{todonotes}
\begin{document}
\begin{abstract}
  We establish a generalized Rieffel correspondence for ideals
  in equivalent Fell bundles.
\end{abstract}

\title{The Rieffel Correspondence for Equivalent Fell Bundles}

\author[Kaliszewski]{S. Kaliszewski}
\address{School of Mathematical and Statistical Sciences
\\Arizona State University
\\Tempe, Arizona 85287}
\email{kaliszewski@asu.edu}

\author[Quigg]{John Quigg}
\address{School of Mathematical and Statistical Sciences
\\Arizona State University
\\Tempe, Arizona 85287}
\email{quigg@asu.edu}

\author[Williams]{Dana P. Williams}
\address{Department of Mathematics\\ Dartmouth College \\ Hanover, NH
  03755-3551 USA}
\email{dana.williams@Dartmouth.edu}

 \date{September 18, 2023}


\maketitle

\section{Introduction}
\label{sec:introduction}

Morita equivalence is a fundamental tool in the
study of \cs-algebras.  For example Morita equivalent \cs-algebras
$A$ and $B$ share much of their fine structure and have equivalent
representation theories.  Many such properties
are elucidated as the ``Rieffel Correspondence'' induced by an
$A\sme B$-\ib\ $\X$.  A summary of these properties is given in
Theorem~\ref{thm-rief-corr} below, but the key feature is that the
Rieffel Correspondence gives a natural lattice isomorphism between the
ideal lattices of the two \cs-algebras.   In
the case of \cs-algebras associated to dynamical systems of various
sorts, perhaps the fundamental tool used to generate useful Morita
equivalences is the notion of Fell-bundle equivalence.  In this
article, we show that there is an analogous Rieffel correspondence
induced by an equivalence $q\colon \E\to T$ between two Fell bundles
$\pb\colon \B\to H$ and $\pc\colon \CC\to K$ over locally compact
groupoids $H$ and $K$.  Rather than work at the level of the Fell-bundle
\cs-algebras $\cs(H; \B)$ and $\cs(K;\CC)$, we work with the
Fell bundles themselves.  We introduce a natural notion of an ideal
$\JJ$ of a Fell bundle $\B$.  In the case where $\B$ is the Fell
bundle corresponding to a group or groupoid $G$ acting on a
\cs-algebra $A$, these Fell-bundle ideals naturally correspond to
$G$-invariant ideals of $A$ in the standard sense.  More generally, our
ideals are the same as the Fell subbundles studied in
\cite{ionwil:hjm11}.  We can form the quotient Fell bundles $\B/\JJ$,
and if $H$ has a Haar system and if our Fell bundles are separable,
then the main result in \cite{ionwil:hjm11} pushes the analogy of Fell
bundle ideals with invariant ideals in crossed products; that is,
we have a short
exact sequence of \cs-algebras
\begin{equation}
  \label{eq:15}
  \begin{tikzcd}
    0\arrow[r]&\cs(H;\JJ) \arrow[r]
    &\cs(H;\B)\arrow[r]&\cs(H;\B/\JJ)\arrow[r]
    &0.
  \end{tikzcd}
\end{equation}
However, as we do not work with \cs-algebras, we do not require our
groupoids to have Haar systems.

In this article, our main result is that if $\E$ is an equivalence
between $\B$ and $\CC$ as above, then there is a lattice isomorphism
between the ideals of $\CC$ and those of $\B$.  Furthermore, if $\KK$
and $\JJ$ are corresponding ideals of $\B$ and $\CC$, respectively,
then $\KK$ and $\JJ$ are equivalent Fell bundles as are the quotients
$\B/\KK$ and $\CC/\JJ$.  Naturally, these equivalences arise from
submodules and quotients of the given equivalence $\E$.

We start in Section~\ref{sec:preliminaries} with a detailed collection
of preliminary material which summarizes and conveniently
collects in one place the
basics of Banach bundles, Fell bundles, and Fell-bundle equivalence.
We also introduce our notion of ideals of Fell bundles and develop some
of their basic properties.    In Section~\ref{sec:rieff-corr-fell}, we
establish our basic Rieffel Correspondence as
Theorem~\ref{thm-rieffel-corr-fell}.   Then in
Section~\ref{sec:extend-rieff-corr} we establish the equivalence
between corresponding ideals and their quotients.

Since we know that if two separable Fell
bundles are equivalent, and if both of the underlying groupoids have
Haar systems, then their corresponding Fell-bundle
\cs-algebras are Morita equivalent, the classical Rieffel
correspondence gives an isomorphism between the ideal lattices of the
two Fell-bundle \cs-algebras.   In 
Section~\ref{sec:at-cs-level}, we confirm the
natural conjecture that if
two ideals correspond under our Rieffel correspondence for Fell-bundle
ideals, then the corresponding ideals in the Fell-bundle \cs-algebras
also correspond under the classical Rieffel correspondence.

\subsection*{Conventions}
\label{sec:conventions}

We use the standard conventions in the subject.  In particular,
homomorphisms between \cs-algebras are assumed to be $*$-preserving
and ideals in \cs-algebras are two-sided and norm closed.
Locally compact is meant to mean locally compact and Hausdorff, and
our groupoids are always meant to be locally compact and Hausdorff.
Suppose that $A$ is an algebra and $\X$ is a (left) $A$-module.  If
$S\subset A$ and $\Y\subset \X$, then by convention,
$S\cdot \Y=\Span\set{a\cdot x:\text{$a\in A$ and $x\in \Y$}}$.
Similarly, if $\rip<\cdot, \cdot>$ is an $A$-valued sesquilinear form
on $\X$, then
$\rip<\Y_{1},\Y_{2}>=\Span\set{\rip<x,y>:\text{$x\in \Y_{1}$ and
    $y\in \Y_{2}$}}$.  If $A$ is a \cs-algebra and the $A$-module $\X$
is a Banach space, then we call $\X$ a Banach $A$-module if
$\|a\cdot x\|\le \|a\|\|x\|$ for all $a\in A$ and $x\in \X$.  Further,
we say that $\X$ is nondegenerate if $A\cdot \X$ is dense in~$\X$.

\section{Preliminaries}
\label{sec:preliminaries}

\subsection{The Rieffel Correspondence}
\label{sec:rieff-corr}

If $A$ is a \cs-algebra, then we let $\I(A)$ denote the lattice of
ideals in $A$.
Suppose that $\X$ is an $A\sme B$-\ib, and let $\mathcal C(\X)$ be the
lattice of closed $A\sme B$-submodules of $\X$.  Then the Rieffel
Correspondence asserts that there are
natural lattice isomorphisms among 
$\I(A)$, $\mathcal C(\X)$, and $\I(B)$.  Specifically, we have the
following summary 
from \cite{rw:morita}*{\S3.3}.

\begin{thm}[Rieffel Correspondence]  \label{thm-rief-corr}
  Suppose that $A$ and $B$ are \cs-algebras and that
  $\X$ is an $A\sme B$-\ib.
  \begin{enumerate}
  \item Suppose that $\Y$ is a closed $A\sme B$-submodule of $\X$.  Then
  \begin{equation}
    \label{eq:3}
    K=\overline{\lip A<\Y,\X>}=\overline{\lip A<\X,\Y>} =\overline{\lip
      A<\Y,\Y>} 
  \end{equation}
  is an ideal in $A$, while
  \begin{equation}
    \label{eq:4}
    J=\overline{\rip B<\Y,\X>}=\overline{\rip B<\X,\Y>} =\overline{\rip
      B<\Y,\Y>} 
  \end{equation}
is an ideal in $B$.   We have
\begin{equation}
  \label{eq:5}
 K\cdot \X= \overline{K\cdot \X}=\Y=\overline{\X\cdot J}= \X\cdot J.
\end{equation}
\item
In particular, $J\mapsto \X\cdot J$ is a lattice isomorphism of
$\I(B)$ onto $\mathcal C(\X)$ with inverse $\Y\mapsto \overline{\rip
  B<\Y,\Y>}$, and $K\mapsto K\cdot \X$ is a lattice isomorphism of
$\I(A)$ onto $\mathcal C(\X)$ with inverse $\Y \mapsto \overline{\lip
  A<\Y,\Y>}$.
\item If $K$, $\Y$, and $J$ are as in part~(a), then $\Y$ is a $K\sme
  J$-\ib\ with respect to the restricted actions and inner products.
\item If $K$, $\Y$, and $J$ are as in part~(a), then the quotient
  Banach space $\X/\Y$ is an $A/K\sme B/J$-\ib.   In particular, the
  quotient norm on $\X/\Y$ equals the imprimitivity-bimodule norm.
  \end{enumerate}
\end{thm}

\begin{remark}
  \label{rem-cohen} Suppose that $J$ is an ideal in a \cs-algebra
  $B$, and that $\X$ is a right
  Hilbert $B$-module.  Then $\Y=\overline{\X\cdot J}$ is a
  nondegenerate Banach $J$-module.   Therefore the Cohen factoriztion
  \cite{rw:morita}*{Proposition~2.33} implies that every element of
  $\Y$ is of the form $x\cdot b$ with $x\in \Y$ and $b\in J$,
  so
  \begin{equation}
    \Y=\X\cdot J=\set{x\cdot b:\text{$x\in \X$ and $b\in J$}}.
  \end{equation}
As a result, we have $\overline{\X\cdot J}=\X\cdot J$ part~(a), and
similarly with $K\cdot \X$. 
\end{remark}

\begin{proof}[Proof of Theorem~\ref{thm-rief-corr}]
  This is just a reworking of the basic results in
  \cite{rw:morita}*{\S3.3}.  The equalities in \eqref{eq:3} and
  \eqref{eq:4} follow from \cite{rw:morita}*{Lemma~3.23}.  The lattice
  isomorphisms follow from \cite{rw:morita}*{Theorem~3.22} while
  \eqref{eq:5}
  follows from \cite{rw:morita}*{Proposition~3.24} together with
  Remark~\ref{rem-cohen}.  The statements about \ib s follow from
  \cite{rw:morita}*{Proposition~3.25}. 
\end{proof}

\subsection{Banach Bundles}
\label{sec:banach-bundles}

Roughly speaking, a Banach bundle is a topological bundle in which
each fibre is a Banach space.   More precisely, we have the following
definition.

\begin{definition} \label{def-ban-bun} A \emph{Banach
  bundle} over a 
  topological space $X$ is a topological space $\B$ together with a
  continuous, \emph{open} surjection $p\colon \B\to X$ and complex Banach
  space structures on each fibre $B_{x}=p^{-1}(\sset x)$ satisfying
  the following axioms.
  \begin{enumerate}[(B1)]
  \item The map $b\mapsto \|b\|$ is upper semicontinuous from $\B$ to
    $\R^{+}$.\footnote{This means that for all $\epsilon>0$,
      $\set{b\in \B:\|b\|<\epsilon}$ is open.}
  \item The map $(a,b)\mapsto a+b$ is continuous from
    $\B^{(2)}=\set{(a,b)\in \B\times\B:p(a)=p(b)}$ to $\B$.
  \item The map $(\lambda,b)\mapsto \lambda b$ is continuous
    from $\C\times \B$ to $\B$.
  \item If $(b_{i})$ is a net in $\B$ such that $p(b_{i})\to x$ and
    $\|b_{i}\|\to 0$, then $b_{i}\to 0_{x}$ in $\B$ (where $0_{x}$ is
    the zero element in $B_{x}$.
  \end{enumerate}
  We say that $p:\B\to X$ is \emph{separable} if $X$ is second
  countable and the Banach space $\Gamma_{0}(X;\B)$ is separable.
  If the map in (B1) is actually continuous, we we call $\B$ a
  \emph{continuous} Banach bundle.
\end{definition}

\begin{remark}
  In some treatments axiom~(B3) in Definition~\ref{def-ban-bun} is
  replaced by the formally weaker axiom 
  that $b\mapsto \lambda b$ is continuous for each $\lambda\in\C$.
  However since $\set{b\in\B:\|b\|<\epsilon}$ is open, the proof of
  \cite{fd:representations1}*{Proposition~II.13.10} shows the two
  definitions are equivalent.
\end{remark}

\begin{remark}[The Literature]
  Continuous Banach bundles are treated in detail in \S\S13--14 in
  \cite{fd:representations1}*{Chap.~II} and many of the results there
  apply \emph{mutatis mutandis} to Banach bundles. In the
  past, Banach bundles as defined above were called ``upper
  semicontinuous Banach bundles''.  We have adopted the convention to
  drop the modifier in the general case.
  Banach bundles are
  discussed briefly in \cite{muhwil:dm08}*{Appendix~A}---which is
  where Definition~\ref{def-ban-bun} comes from---and the
  corresponding notion of a \cs-bundle is treated in detail in
  \cite{wil:crossed}*{Appendix~C}.
\end{remark}

The topology on the total space $\B$ of a Banach bundle might not be well
behaved.  For example, it need not be Hausdorff
\cite{wil:crossed}*{Example~C.27}.  But we do have the following.

\begin{lemma}
  \label{lem-rel-top} If $p\colon \B\to X$ is a Banach
  bundle, then the relative topology on $B_{x}$ is the \textup(Banach space\textup)
  norm topology. 
\end{lemma}

\begin{proof}
  In the continuous case, this is
  \cite{fd:representations1}*{Proposition~II.13.11}, and proof carries
  over to the general case---see \cite{dwz:jmaa22}*{Lemma~2.2}.
\end{proof}

If $p:\B\to X$ is a Banach bundle, we write $\Gamma(X;\B)$ for the
vector space of continuous sections.  If $X$ is locally
compact, then
we write $\Gamma_{c}(X;\B)$ and $\Gamma_{0}(X;\B)$ for the continuous
sections which have compact support or which vanish at infinity, respectively.
We say that $p:\B\to X$ \emph{has enough sections} if given
$b\in B_{x}$ there is  an $f\in \Gamma(X;\B)$ such that $f(x)=b$.  Note
that if $X$ is locally compact, then since $\Gamma(X;\B)$ is a
$C(X)$-module by (B3), we can take $f\in \Gamma_{c}(X;\B)$.

\begin{thm}[\cite{laz:jmaa18}*{Corollary~2.10}]
  \label{thm-enough-sections} If $p:\B\to X$ is a Banach bundle over a
  locally compact space, then $\B$ has enough sections.
\end{thm}

While the notion of a Banach bundle is a natural mathematical object,
generally Banach bundles arise in nature from their sections as
described in the following result.

\begin{thm}[Hofmann-Fell] \label{thm-section-top}
  Let $X$ be a locally compact space and suppose that for each $x\in
  X$, we are given a Banach space $B_{x}$.   Let $\B$ be the disjoint
  union $\coprod_{x\in X}B_{x}$ viewed as bundle $p:\B\to X$.
  Suppose that $\Gamma$ is a subspace of sections such that
  \begin{enumerate}
  \item for each $f\in \Gamma$, $x\mapsto \|f(x)\|$ is upper
    semicontinuous, and 
  \item for each $x\in X$, $\set{f(x):f\in\Gamma}$ is dense in $B_{x}$.
  \end{enumerate}
Then there is a unique topology on $\B$ such that $p:\B\to X$ is a
Banach bundle with $\Gamma\subset \Gamma(X;\B)$.  Furthermore, the
sets of the form
\begin{equation}
  \label{eq:1}
  W(f,U,\epsilon)=\set{a\in \B:\text{$p(a)\in U$ and
      $\|a-f(p(a))\|<\epsilon$}} 
\end{equation}
 with $f\in \Gamma$, $U$ open in $X$, and $\epsilon>0$ form a basis
 for this topology.
\end{thm}
\begin{proof}
  In the continuous case, this is
  \cite{fd:representations1}*{Theorem~II.13.18}.  In general, it was
  stated in \cite{dg:banach}*{Proposition~1.3} and also follows
  \emph{mutatis mutandis} from \cite{wil:crossed}*{Theorem~C.25}. 
\end{proof}

\subsection{Banach Subbundles}
\label{sec:banach-subbundles}

A subbundle of Banach bundle is a Banach subbundle if it is a Banach
bundle in the inherited structure.

\begin{definition}\label{def-ban-subbun}
  Let $p\colon \B\to X$ be a Banach bundle.  We say that $\CC\subset
  \B$ is a \emph{Banach subbundle} if each $C_{x}=B_{x}\cap C$
  is a closed vector subspace of $B_{x}$, and $p\restr \CC\colon
  \CC\to X$ is a
  Banach bundle when we give $C_{x}$ the Banach-space
  structure
  coming from $B_{x}$ and we give $\CC$ the relative topology.
\end{definition}

\begin{remark}
  \label{rem-zero-fibres}  Since
  $0_{x}\in C_{x}$ for all $x$, we must have $p(\CC)=X$.
  However some fibres can be the zero Banach
  space.   
\end{remark}

\begin{remark}\label{rem-just-open}
  If $\set{C_{x}}$ is any collection of closed subspaces with
  $C_{x}\subset
  B_{x}$,
  and if we give $\CC=\coprod C_{x}=\set{b\in \B:b\in C_{p(b)}}$ the
  relative topology, then $p\colon \CC\to X$ is a continuous surjection
  satisfying (B1), (B2), (B3), and (B4) of 
  Definition~\ref{def-ban-bun}.
  But $p\colon \CC\to X$ may fail to be a Banach subbundle unless we
  also have $p\restr \CC$ open.
\end{remark}

Even if $p\restr \CC$ is not open, we write $\Gamma(X;\CC)$ for the
continuous functions $f$ from $X$ to $\CC$ such that $p(f(x))=x$ for
all $x\in X$.  Of course, if $p\restr \CC$ is not open, there is no
reason that $\Gamma(X;\CC)$ should contain anything other than the zero
section---as shown in the next example.

\begin{example}
  Let $B$ be a Banach space and $\B=X\times B$ the trivial bundle over
  $X$.   Fix $x_{0}\in X$, and let
  \begin{equation}
    \label{eq:8}
    C_{x}=
    \begin{cases}
      B&\text{if $x=x_{0}$, and} \\ 0_{x}&\text{otherwise.}
    \end{cases}
  \end{equation}
Then in general, $p\restr \CC:\CC\to X$ is not open and admits only the
zero section.
\end{example}

\begin{prop}
  \label{prop-fd-sub-bun} Let $p\colon \B\to X$ be a Banach bundle.
  Suppose that $C_{x}$ is a closed subspace of $B_{x}$ for all $x\in X$
  and let $\CC=\coprod C_{x}$ be as above.   If $\set{f(x):f\in
    \Gamma(X;\CC)}$ is dense in $C_{x}$ for all $x\in X$, then the
  bundle $p\restr 
  \CC\colon \CC\to X$ is a Banach subbundle of~$\B$.
\end{prop}
\begin{proof}
  Suppose that $\set{f(x):f\in \Gamma(X;\CC)}$ is dense in $C_{x}$ for
  all $x$.
  Let $U$ be a nonempty (relatively) open set in $\CC$.   In view of
  Remark~\ref{rem-just-open}, to show that $p\restr \CC\colon \CC\to X$ is a
  Banach subbundle, it will suffice to see that $p(U)$ is
  open in $X$.   Let $x\in p(U)$ and suppose that $(x_{i})$ is a net
  in $X$ converging to $x$ in $X$.   It will suffice to see that
  $(x_{i})$ is eventually in $p(U)$.  
  Let $b\in U$ be such that $p(b)=x$.   Then for each~$n$, $\set{b'\in
    \B: \|b'-b\|<\frac1n}$ is an  open neighborhood of $b$ in $\B$.
  Hence there is 
  $f_{n}\in \Gamma(X;\CC)$ such that $\|f_{n}(x)-b\|<\frac1n$.  Thus
  $\|f_{n}(x)-b\|\to 0$.   By B4, $f_{n}(x)-b\to 0_{x}$ in $\B$.   But
  by 
  B2, $f_{n}(x) \to b$ in $\B$.  Since everything in sight is in $\CC$
  and $\CC$ has the relative topology, for some $N$, $f_{N}(x)\in U$.
  But $f_{N}(x_{i})\to f_{N}(x)$.   So $f_{N}(x_{i})$ is eventually in
  $U$.   But then $x_{i}$ is eventually in $p(U)$.
\end{proof}

\begin{remark}
  In \cite{fd:representations1}*{Problem 41 in Chap. II}, Fell and
  Doran call a family $\set{C_{x}}$ of subspaces 
  as in Proposition~\ref{prop-fd-sub-bun} in a
  continuous Banach bundle a
  \emph{lower semicontinuous choice of subspaces}.
\end{remark}

\begin{remark}
  If $p\colon \B\to X$ is a Banach bundle over a locally compact space,
  and if $p\restr\CC\colon \CC\to X$ is a Banach subbundle, then it
  has enough sections by Lazar's Theorem~\ref{thm-enough-sections}.
  Hence $\set{f(x):f\in \Gamma(X;\CC)}$ is not only dense---it is all
  of $C_{x}$.
\end{remark}


\subsection{Quotient Banach Bundles}
\label{sec:quot-banach-bundl}
Let $p:\B\to X$ be a Banach bundle over a locally compact space $X$,
and let $\CC\subset \B$ be a Banach subbundle as in
Definition~\ref{def-ban-bun}.  Then we can formally form the quotient
$\B/\CC=\coprod_{x\in X} B_{x}/C_{x}$ where $B_{x}/C_{x}$ is the usual
Banach space quotient.  We let $q:\B\to \B/\CC$ be the quotient map so that
if $b\in B_{x}$, then $q(b)=q_{x}(b)$ where
$q_{x}:B_{x}\to B_{x}/C_{x}$ is the usual Banach space quotient map.
In particular, $q$ is norm reducing.  If $f\in \Gamma_{c}(X;\B)$ then
we will write $q(f)$ for the section of $\B/\CC$ given by
$q(f)(x)=q_{x}(f(x))$. 
\begin{prop}
  \label{prop-ban-bund-quotient}  Let $p:\B\to X$ be a Banach bundle
  and $\CC\subset \B$ a Banach subbundle.   Then $\bar p:\B/\CC \to X$
  is a Banach bundle in the quotient topology.   Furthermore, the
  quotient map $q:\B\to \B/\CC$ is continuous and open, and the
  quotient topology on $\B/\CC$ is the unique topology such that
  $\Gamma= \set{q(f):f\in \Gamma_{c}(X;\B)} \subset
  \Gamma_{c}(X;\B/\CC)$. 
\end{prop}

\begin{remark}
  As pointed out in \cite{laz:jmaa18},
  Proposition~\ref{prop-ban-bund-quotient} can be sorted out of
  \cite{gie:lnm82}*{Chap.~9}.   We give the short proof for
  completeness. 
\end{remark}

\begin{proof}
  Let $f\in \Gamma_{c}(X;\B)$.   We claim $x\mapsto \|q_{x}(f(x))\|$
  is upper semicontinuous.   Fix $\epsilon>0$.  Suppose
  $\|q_{x}(f(x))\|<\epsilon$.   Then by definition of the quotient
  norm, there is a $c\in C_{x}$ such that $\|f(x)+c\|<\epsilon$.   Let
  $d\in \Gamma_{c}(X;\CC)$ be such that $d(x)=c$.   Then there is a
  neighborhood $V$ of $x$ such that $\|f(y)+d(y)\|<\epsilon$ if $y\in
  V$.   Since $q(f)=q(f+d)$, it follows that $\|q(f)(y)\|<\epsilon$
  for $y\in V$.   This establishes the claim.

  It follows from Theorem~\ref{thm-section-top} that there is a
  unique topology on $\B/\CC$ such that $\B/\CC$ is a Banach bundle
  with $\Gamma:=\set{q(f):f\in \Gamma_{c}(X;\B)} \subset
  \Gamma_{c}(X;\B/\CC)$.

  Next we claim that the quotient map $q:\B\to \B/\CC$ is continuous.
  Suppose that $(a_{i})$ is a net in $\B$ such that $a_{i}\in
  B_{x_{i}}$ and $a_{i}\to a_{0}$ in $\B$.   Then $x_{i}\to x_{0}$ in
  $X$.  Let $f\in\Gamma_{c}(X;\B)$ be such that $f(x_{0})=a_{0}$.   Then
  \begin{equation}
    \label{eq:37}
    \|f(x_{i})-a_{i}\|\to 0.
  \end{equation}
  Since $q$ is norm reducing,
  \begin{equation}
    \label{eq:48}
    \|q(f)(x_{i})-q(a_{i})\|\to 0.
  \end{equation}
  Since $q(f)\in \Gamma_{c}(X;\B/\CC)$, \cite{muhwil:dm08}*{Lemma~A.3}
  implies that $q(a_{i})\to q(a_{0})$.   Thus $q$ is continuous as
  claimed.

  To see that $q$ is also open, Let $V$ be an open neighborhood of
  $b\in \B$.   Then in view of Theorem~\ref{thm-section-top}, there is a
  $f\in \Gamma_{c}(X;\B)$, and open neighborhood $U$ of $p(b)$, and a
  $\epsilon>0$ such that
  \begin{equation}
    \label{eq:49}
    b\in W(f,U,\epsilon):=\set{a\in B:\text{$p(a)\in U$ and
        $\|a-f(p(a))\|<\epsilon$}}. 
  \end{equation}
We need to verify that $q(V)$ is a neighborhood of $q(b)$.  Since
$q(f) \in \Gamma_{c}(X;\B/\CC)$, it will suffice to see that
\begin{equation}
  \label{eq:50}
  q(W(f,U,\epsilon))=\set{q(c):\text{$p(c)\in U$ and
      $\|q(c)-q(f)(p(c))\|<\epsilon$} }.
\end{equation}
Since the left-hand side is clearly a subset of the right-hand side,
it suffices to consider $q(c)$ in the right-hand side.   If $x=p(c)$,
then
\begin{equation}
  \label{eq:51}
  \epsilon>\|q(c)-q(f)(x)\|=\|q_{x}(c-f(x))\|=\inf_{q_{x}(a)=c}\|a-f(x)\|. 
\end{equation}
Hence there is an $a\in W(f,U,\epsilon)$ such that $q(a)=q(c)$.   This
suffices show that $q$ is open.

Since $q$ is continuous and open, the topology on $\B/\CC$  is the
quotient topology.
\end{proof}

\subsection{Fell Bundles}
\label{sec:fell-bundles}

Fell bundles are natural generalizations of Fell's Banach
$*$-algebraic bundles from \cite{fd:representations1}*{Chap.~VIII} and
were introduced by Yamagami in \cite{yam:xx87}.  The following
definition comes from \cite{muhwil:dm08}*{Definition~1.1}.

\begin{definition}
  Suppose that $p:\B\to G$ is
  a Banach bundle over a second countable locally compact Hausdorff
  groupoid $G$.  Let
  \begin{equation}
    \label{eq:6}
    \B^{(2)}=\set{(a,b)\in \B\times \B:\bigl(p(a),p(b)\bigr)\in G^{(2)}}.
  \end{equation}
We say that $p:\B\to G$ is a \emph{Fell bundle} if there is a
continuous, bilinear, associative multiplication map $(a,b)\mapsto ab$
from $\B^{(2)}$ to $\B$ and a continuous involution $b\mapsto b^{*}$
from $\B$ to $\B$ such that
\begin{enumerate}[(FB1)]
\item $p(ab)=p(a)p(b)$,
\item $p(a^{*})=p(a)^{-1}$,
\item $(ab)^{*}=b^{*}a^{*}$,
\item for each $u\in\go$, the fibre $B_{u}$ is a \cs-algebra with
  respect to the inherited multiplication and involution on $B_{u}$,
  and
\item for each $g\in G$, $B_{g}$ is an $B_{r(g)} \sme B_{s(g)}$-\ib\
  when equipped with the inherited actions and inner products given by
  \begin{equation}
    \label{eq:7}
    \lip B_{r(g)} <a,b>=ab^{*}\quad\text{and} \quad \rip B_{s(g)}
    <a,b>=a^{*}b. 
  \end{equation}
\end{enumerate}
We say that the Fell bundle
$p:\B\to G$ is \emph{separable} if it is separable as a
Banach bundle.
\end{definition}

\begin{remark}[Saturated]
  It should be noted that our Fell bundles are saturated in that
  whenever $(g,h)\in G^{(2)}$, then
  $B_{g}\cdot B_{h}:=\operatorname{span}\set{ab:\text{$a\in B_{g}$ and
    $b\in B_{h}$}} $ is always dense in $B_{gh}$
\cite{muhwil:dm08}*{Lemma~1.2}.  This is a consequence of (FB5).
Some authors prefer to work with a weakened version of (FB5) where the
inner products in \eqref{eq:7} are not full.
\end{remark}

\begin{remark}
  \label{rem-assoc-cs} If $p\colon\B\to G$ is a Fell bundle, then the
  restriction $\B\restr\go$ is a \cs-bundle and $\Gamma_{0}(\go;\B)$
  is a \cs-algebra called the \emph{associated \cs-algebra} to
  $\B$.\footnote{The terminology is a bit challenging.  If $G$ has a
    Haar system, then one can also form the Fell-bundle \cs-algebra
    $\cs(G;\B)$ by viewing $\Gamma_{c}(G;\B)$ as a $*$-algebra and
    completing as in \cite{muhwil:dm08}.  
  }
\end{remark}

\subsection{Equivalence of Fell Bundles}
\label{sec:equiv-fell-bundl}

Suppose that $T$ is a left $G$-space.   Then we say that a Fell bundle
$p\colon \B\to G$ acts on (the left of) a Banach bundle $q:\E\to T$ if there is a
continuous map $(b,e) \mapsto b\cdot e$ from
$\B*\E:=\set{(b,e)\in\B\times\E:s(b)=r(q(e))}$ to $\E$ such that
\begin{enumerate}
\item $q(b\cdot e)=p(b)\cdot q(e)$,
\item $a\cdot (b\cdot e)=(ab)\cdot e$ for appropriate $a,b\in B$ and
  $e\in \E$, and
\item $\|b\cdot e\|\le \|b\|\|e\|$.
\end{enumerate}
Right actions of a Fell
bundle are defined similarly.

Let $T$ be a $(G,H)$ equivalence with open moment maps $\rho\colon T\to \go$
and $\sigma\colon T\to\ho$ as in \cite{wil:toolkit}*{Definition~2.29}.
It is shown in \cite{wil:toolkit}*{Lemma~2.42} that there  are open
continuous maps $\tau_{G}\colon T*_{\sigma}T\to 
G$ and $\tau_{H}\colon T*_{\rho}T\to H$ such that $\tau_{G}(e,f)\cdot f=e$
and $e\cdot \tau_{H}(e,f)=f$.

\begin{definition}
  [\cite{muhwil:dm08}*{Definition~6.1}]   \label{def-equi}
  Suppose the $T$ is a
  $(G,H)$-equivalence, and that $\pb\cdot \B\to G$ and $\pc\colon\CC\to H$
  are Fell bundles.  Then a Banach bundle $q\colon \E\to T$ is a
  \emph{$\B\sme\CC$-equivalence} if the following conditions hold.
  \begin{enumerate}[(E1)]
  \item There is a left $\B$-action and a right $\CC$-action on $\E$
    such that $b\cdot (e\cdot c)=(b\cdot e)\cdot c$ for composable
    $b\in \B$, $e\in E$, and $c\in \CC$.
  \item There are continuous sesquilinear maps $(e,f)\mapsto
    \lip\B<e,f>$ from $\E*_{\sigma}\E$ to $\B$ and $(e,f)\mapsto
    \rip\CC<e,f>$ from $\E*_{\rho}\E$ to $\CC$ such that
    \begin{enumerate}
    \item $\pb\bigl(\lip\B<e,f>\bigr) =\tau_{G}(q(e),q(f))$ and
      $\pc\bigl(\rip\CC<e,f>\bigr) =\tau_{H}(q(e),q(f))$,
    \item $\lip\B<e,f>^{*} =\lip\B<f,e>$ and $\rip\CC<e,f>^{*} =
      \rip\CC<f,e>$, 
    \item $\lip\B<b\cdot e,f>=b\lip\B<e,f>$ and $\rip\CC<e,f\cdot c> =
      \rip\CC<e,f>c$, and 
    \item $\lip\B <e,f>\cdot g=e\cdot \rip\CC<f,g>$.
    \end{enumerate}
\item With the actions and inner products coming from (E1) and (E2),
  each $E_{t}$ is a $B_{\rho(t)} \sme C_{\sigma(t)}$-\ib.
  \end{enumerate}

\end{definition}

\subsection{Fell Subbundles and Ideals}
\label{sec:fell-subb-ideals}

Naturally, a subbundle of a Fell bundle is called a Fell subbundle if
it is a Fell bundle in the inherited structure.

\begin{definition}\label{def-sub-bundle}
  Let $p\colon \B\to G$ be a Fell bundle over a groupoid $G$.   We call
  $\CC\subset \B$ a \emph{Fell subbundle}  if $\CC$ is a Banach
  subbundle such that $p\restr \CC\colon \CC\to G$ is a Fell bundle with
  respect to the inherited operations.  In particular, $\CC$ must be
  closed under multiplication and involution.
\end{definition}

We will be focused on  Fell subbundles that are multiplicatively
absorbing as follows.

\begin{definition}
  \label{def-fb-ideals} A Fell subbundle $\JJ$ of a Fell bundle $\B$
  is called an \emph{ideal} if $ab\in \JJ$ whenever $(a,b)\in
  \B^{(2)}$ and either $a\in \JJ$
  or $b\in \JJ$.
\end{definition}

\begin{example}\label{ex-std-ideal}
  Suppose that $\alpha\colon \grg\to \operatorname{Aut}(A)$ is a \cs-dynamical
  system for a group~$\grg$.  Let $\B=A\times \grg$ be the associated Fell
  bundle over $\grg$: $(a,s)(b,r)=(a\alpha_{s}(b),sr)$.  Let $I$ be an
  \emph{$\alpha$-invariant ideal} of $A$.  Then $\JJ=I\times \grg$ is an
  ideal in $\B$.
\end{example}

\begin{example}
  \label{ex-cox-alg}  Suppose that $\A$ is a \cs-bundle over $X$
  so that $A=\Gamma_{0}(X;\A)$ is a \cs-algebra.   Let $J$ be an ideal
  in $A$ and for each $x\in X$ let $J_{x}=\set{a(x):a\in J}$ so that $J_{x}$ is an ideal
  in $A_{x}$.    Let
  \begin{equation}
    \label{eq:11} \JJ=\coprod_{x\in X}J_{x}.
  \end{equation}
  It follows from Proposition~\ref{prop-fd-sub-bun} that $\JJ$ is a Banach
  subbundle, and in fact is obviously an ideal of the Fell bundle $\A$.  
  Clearly, $J\subset \Gamma_{0}(X;\JJ)$.  Since $J$ is an
  ideal in the $C_{0}(X)$-algebra $A$, if $\phi\in C_{0}(X)$ and
  $b\in J$, then $\phi\cdot b\in J$.  Now it follows from
  \cite{wil:crossed}*{Proposition~C.24} that $J$ is dense in
  $\Gamma_{0}(X;\JJ)$.  Therefore $J=\Gamma_{0}(X;\JJ)$.
\end{example}
\begin{lemma}\label{lem-ib-cohen}
  Suppose that $\JJ$ is an ideal in $\B$.  Then for each $g\in G$,
  $J_{g}$ is a $J_{r(g)} \sme J_{s(g)}$-\ib.   Furthermore,
  \begin{equation}
    \label{eq:9a}
    J_{g}=B_{g}\cdot
  J_{s(g)} = J_{r(g)}\cdot B_{g} 
\end{equation}
where we are taking advantage of Remark~\ref{rem-cohen}.
Furthermore,
\begin{equation}
  \label{eq:20}
  J_{r(g)}=\overline{J_{g}B_{g}^{*}}=\overline{B_{g}J_{g}^{*}}
  =\overline{J_{g}J_{g}^{*}} \quad\text{and} \quad
  J_{s(g)} =\overline{J_{g}^{*}B_{g}}=\overline{B_{g}^{*}J_{g}} =
  \overline{J_{g}^{*}J_{g}}. 
\end{equation}
\end{lemma}
\begin{proof}
  The first assertion is immediate since $\JJ$ is, by assumption, a
  Fell subbundle.  The remaining statements follow from the 
  Rieffel correspondence---see Theorem~\ref{thm-rief-corr}. 
\end{proof}

\begin{definition}\label{def-weak-ideal}
  Let $p:\B\to G$ be a Fell bundle and $\JJ\subset \B$ a Banach
  subbundle.  We call $\JJ$ a \emph{weak ideal of $\B$} if whenever
  $(a,b)\in \B^{(2)}$ then $ab\in \JJ$ whenever either $a$ or $b$ is
  in $\JJ$.
\end{definition}

\begin{remark}\label{rem-weak-ideal}
  If $I$ is an ideal in a \cs-algebra, then the existence of
  approximate identities implies that $I$ is $*$-closed and hence a
  \cs-subalgebra.   A similar serendipity applies to weak ideals.
\end{remark}

\begin{prop}
\label{prop-weak-ideal}  If $p:\B\to G$ is a Fell bundle, then every
weak ideal in $\B$ is an ideal.
\end{prop}

\begin{proof}
  Let $\JJ$ be a weak ideal in $\B$.  
  Since $J_{g}$ is closed with respect to the norm on $B_{g}$, it is a 
closed $B(r(g))\sme B(s(g))$-submodule of the $B(r(g))\sme
B(s(g))$-\ib\ $B_{g}$.   In particular, $J_{u}$ is an ideal in the
\cs-algebra $B_{u}$ for all $u\in \go$.
Then, applying the Rieffel correspondence, $J_{g}$ is a $K_{g}\sme
I_{g}$-\ib\ where $I_{g}$ is the closed linear span of elements of the
form $b^{*}a$ with
$b\in B_{g}$ and $a\in J_{g}$.    Similarly, $K_{g}$ is the closed
linear span of products $ab^{*}$ with $a\in J_{g}$ and $b\in B_{g}$.
Furthermore,
\begin{equation}
  \label{eq:29}
J_{g}=B_{g}\cdot I_{g}=K_{g}\cdot B_{g}.
\end{equation}

Note that $I_{g}$ is an
ideal in $J_{s(g)}$.  Fix $c\in J_{s(g)}$.  If $b\in
B_{g}$, then $bc\in J_{g}$ by the weak ideal property.   Since $B_{g}^{*}\cdot
B_{g}$ is a dense ideal in $B_{s(g)}$, we can find an approximate unit
$(e_{i})$ in $B_{s(g)}$ where each $e_{i}=\sum_{k=1}^{n_{i}}
b_{k}^{*}b_{k}$ with each $b_{k}\in B_{g}$.  But then $e_{i}c$ is in
$I_{g}$ and $e_{i}c\to c$.  Hence $c\in  I_{g}$ and $I_{g}=J_{s(g)}$.
A similar argument shows that $K_{g}=J_{r(g)}$.  

Since $J_{s(g)}$ is an ideal in the \cs-algebra
$B_{s(g)}$, we have
$J_{s(g)}^{*}=J_{s(g)}$.
Hence, using \eqref{eq:29}, we have
\begin{equation}
  \label{eq:5a}
  J_{g}^{*}= (B_{g}\cdot J_{s(g)})^{*}= J_{s(g)}^{*}\cdot B_{g}^{*} =
  J_{r(g^{-1})} \cdot B_{g^{-1}} =J_{g^{-1}}.
\end{equation}
In particular, $\JJ^{*}=\JJ$ and $\JJ$ is closed under taking
adjoints.
  Since $\JJ$ is a weak ideal, it is closed under multiplication and
  we just showed it is also closed under the adjoint operation.   Now
  we just have to observe that it is a Fell bundle.
  But this follows
  from the above discussion and identification of $I_{g}$ with
  $J_{s(g)}$ and $K_{g}$ with $J_{r(g)}$.  
\end{proof}

It is standard to think of a Fell bundle $p\colon \B \to G$ as a
generalized groupoid crossed product of $G$ acting on the associated
\cs-algebra $A:=\Gamma_{0}(\go;\B)$.  As an example of this rubric, it
is shown in \cite{ionwil:hjm11}*{Proposition~2.2} that there is a
natural action of $G$ on $\Prim A$ given as follows.  Note that
$\Prim A$ is naturally fibred over $\go$.  Since $B_{g}$ is a
$B_{r(g)} \sme B_{s(g)}$-\ib, the Rieffel Correspondence induces a
homeomorphism $\phi_{g}:\Prim(B_{s(g)})\to \Prim(B_{r(g)})$
\cite{rw:morita}*{Corollary~3.33}. Then the $G$-action is given by
$g\cdot P_{s(g)}= \phi_{g}(P_{s(g)})$.  Naturally, an ideal $I$ in $A$
is called $G$-invariant if
$\operatorname{hull}(I):=\set{P\in \Prim A:P\supset I}$ is
$G$-invariant.  If $I$ is an ideal in $A$ then we let $I_{u}=q_{u}(I)$
where $q_{u}:\Gamma_{0}(\go;\B)\to B_{u}$ is the evaluation map.

  \begin{prop}[\cite{ionwil:hjm11}]
    \label{prop-ionwil-bi}  Suppose 
  $p\colon\B\to G$ is a Fell bundle
    and that $I$ is a $G$-invariant ideal in the associated \cs-algebra
    $\Gamma_{0}(\go;\B)$.   Then
    \begin{equation}
      \B_{I}:=\set{b\in\B:b^{*}b\in I_{s(b)}}
    \end{equation}
    is an ideal in $\B$.  Conversely, if $\JJ$ is an ideal in $\B$,
    then $I=\Gamma_{0}(\go;\JJ)$ is a $G$-invariant ideal in
    $\Gamma_{0}(\go;\B)$ and $\JJ=\B_{I}$.
  \end{prop}
  \begin{proof}
    It follows from the proof of \cite{ionwil:hjm11}*{Lemma~3.1} that
    an ideal $I\subset \Gamma_{0}(\go;\B)$ is $G$-invariant if and
    only if for all $g$ we have $\phi_{g}(I_{s(g)})=I_{r(g)}$.   By
    the Rieffel correspondence, the latter is equivalent to 
    \begin{equation}
      \label{eq:2} \B_{g}\cdot
    I_{s(b)}=I_{r(b)}\cdot B_{g}\quad\text{for all $g\in
      G$.}
    \end{equation}
    
    Suppose that $I$ is $G$-invariant.   Then it follows from
    \cite{ionwil:hjm11}*{Proposition~3.3} that
    $\JJ:=\B_{I}$ is a Fell 
    subbundle such that \eqref{eq:2} holds.   Suppose that
    $a\in B_{g}$ and
    $b\in B_{h}$ are composable. If $a\in \JJ$, then
    \begin{equation}
      \label{eq:9}
      ab\in I_{r(g)}\cdot B_{g}B_{h} \subset I_{r(gh)} \cdot B_{gh}=J_{gh}.
    \end{equation}
Similarly, if $b\in \JJ$, then
\begin{equation}
  \label{eq:10}
  ab\in B_{g}B_{h}\cdot I_{s(h)} \subset B_{gh}\cdot I_{s(gh)} =J_{gh}
\end{equation}
and $B_{I}$ is an ideal.

Now suppose that $\JJ$ is an ideal in $\B$.   Let
\begin{equation}
  \label{eq:12}
  I=\Gamma_{0}(X;\JJ).
\end{equation}
Then as in Example~\ref{ex-cox-alg}, we have $I_{u}=J_{u}$.   Now it
follows from Lemma~\ref{lem-ib-cohen} and \eqref{eq:2} that $I$ is
$G$-invariant.   Since $\B_{I}$ and $\JJ$ have the same fibres,
clearly $\B_{I}=\JJ$.
\end{proof}

\section{The Rieffel Correspondence for Fell Bundle Equivalences}
\label{sec:rieff-corr-fell}

In this section, we let $\qe\colon\E\to T$ be an equivalence between
$\pb\colon \B \to H$ and $\pc\colon\CC\to K$.  In particular, $T$ is a
$(H,K)$-equivalence and we let $\rho\colon T\to \ho$ and
$\sigma\colon T\to \ko$ be the open moment maps.

We will need the following observation from
\cite{muhwil:dm08}*{Lemma~6.2}.
\begin{lemma}
  \label{lem-6.2}
  As above, let \/$\qe\colon\E\to T$ be a Fell-bundle equivalence
  between $\pb\colon \B\to H$ and $\pc\colon \CC\to K$.  Then
  $(b,e) \mapsto b\cdot e$ induces an \ib\ isomorphism of
  $B_{h} \tensor_{B_{\rho(t)}} E_{t}$ onto $E_{h\cdot t}$.  Similarly,
  $(e,c)\mapsto e\cdot c$ induces an isomorphism between
  $E_{t}\tensor_{C_{\sigma(t)}}C_{k}$ and $E_{t\cdot k}$.
\end{lemma}

\begin{cor}
  \label{cor-ej} Let $\E$, $\B$, and $\CC$ be as above.  Let $\JJ$ be
  an ideal in $\CC$ and $\sigma(t)=r(k)$.  Then
  $\overline{E_{t}\cdot J_{k}}=E_{t\cdot k}\cdot J_{s(k)}$.
\end{cor}
\begin{proof}
  Lemma~\ref{lem-6.2} implies that
  $\overline{E_{t}\cdot C_{k}}=E_{t\cdot k}$.  Therefore by
  Lemma~\ref{lem-ib-cohen} we have
  \begin{align}
    \overline{E_{t}\cdot
    J_{k}} = \overline{E_{t}\cdot C_{k}\cdot J_{s(k)}}=E_{t\cdot
    k}\cdot J_{s(k)}.\qquad\qedhere
  \end{align}
\end{proof}

\begin{definition}
  \label{def-bc-submod} Let $\qe\colon\E\to T$ be an equivalence
  between $\pb\colon\B\to H$ and $\pc\colon\CC\to K$.  Then a Banach
  submodule $\M$ of $\E$ is called a \emph{Banach
    $\B\sme\CC$-submodule} if $B_{h}\cdot M_{t}\subset M_{h\cdot t}$
  whenever $s(h)=\rho(t)$ and $M_{t}\cdot C_{k} \subset M_{t\cdot k}$
  whenever $\sigma(t)=r(k)$.  We say that $\M$ is \emph{full} if
  $\overline{B_{h}\cdot M_{t}}=M_{h\cdot t}$ and
  $\overline{M_{t}\cdot C_{k}}=M_{t\cdot k}$.
\end{definition}

\begin{prop}
  \label{prop-e-j} Let $\E$, $\B$, and $\CC$ be as above.  If $\JJ$ is
  an ideal in $\CC$, then
  \begin{equation}
    \label{eq:13}
    \E\cdot \JJ:= \bigcup_{\set{(t,k):\sigma(t)=r(k)}} E_{t}\cdot J_{k}
  \end{equation}
  is a full Banach $\B\sme\CC$-submodule of $\E$ with
  $(\E\cdot \JJ)_{t}= E_{t}\cdot J_{\sigma(t)}$.  Similarly, if $\KK$
  is an ideal in $\B$, then $\KK\cdot \E$ is a full Banach
  $\B\sme\CC$-submodule with
  $(\KK\cdot \E)_{t} =K_{\rho(t)} \cdot E_{t}$.
\end{prop}
\begin{proof}
  We have 
  \begin{align}
    \label{eq:14}
    \E\cdot \JJ
    &= \bigcup_{\set{(t,k):\sigma(t)=r(k)}} E_{t}\cdot J_{k} \subset
      \bigcup_{\set{(t,k):\sigma(t)=r(k)}} \overline{E_{t}\cdot J_{k}}
    \\
    \intertext{which, by Corollary~\ref{cor-ej}, is}
    &= \bigcup_{\set{(t,k):\sigma(t)=r(k)}} E_{t\cdot k}\cdot J_{s(k)}
      =\bigcup_{t\in T} E_{t}\cdot J_{\sigma(t)}\subset \E\cdot \JJ.
  \end{align}
  Therefore $\E\cdot \JJ=\bigcup_{t\in T} E_{t}\cdot J_{\sigma(t)}$
  and $(\E\cdot \JJ)_{t}=E_{t} \cdot J_{\sigma(t)}$ as claimed.

  In particular, $\E\cdot \JJ$ is a bundle over $T$ with closed fibres
  $E_{t}\cdot J_{\sigma(t)}$.  But if $f\in \Gamma_{c}(T;\E)$ and
  $\phi\in \Gamma_{c}(\ko; \JJ)$, then $\phi\cdot f$ given by
  $\phi\cdot f(t)= f(t)\cdot \phi(\sigma(t))$ is a section in
  $\Gamma_{c}(T;\E\cdot \JJ)$.  Now it follows from
  Proposition~\ref{prop-fd-sub-bun} that $\E\cdot \JJ$ is a Banach
  subbundle of $\E$.

  We still need to see that $\E\cdot \JJ$ is a full
  $\B\sme \JJ$-submodule.  But if $s(h)=\rho(t)$, then
  \begin{align}
    \overline{B_{h}\cdot (\E\cdot \JJ)_{t}}
    &= \overline{B_{h}\cdot (E_{t}\cdot J_{\sigma(t)})} \\
    &= \overline{(B_{h}\cdot E_{t})\cdot J_{\sigma(t)}} \\
    &= E_{h\cdot t}\cdot J_{\sigma(t)}= (\E\cdot \JJ)_{h\cdot t}.
  \end{align}

  On the other hand if $\sigma(t)=r(k)$, then
  \begin{align}
    \overline{(\E\cdot \JJ)_{t}\cdot C_{k}}
    &= \overline{E_{t}\cdot (J_{r(k)}\cdot C_{k})} 
    \\
    \intertext{which, by Lemma~\ref{lem-ib-cohen}, is}
    &= \overline{E_{t}\cdot (C_{k}
      \cdot J_{s(k)})} \\
    &= E_{t\cdot k}\cdot J_{s(k)} = (\E\cdot\JJ)_{t\cdot k}.
  \end{align}
  Thus $\E\cdot \JJ$ is a full Banach $\B\sme\CC$-submodule as
  claimed.
  
  The corresponding statements for $\KK\cdot\E$ are proved similarly.
\end{proof}

Now suppose that $\M$ is a full Banach $\B\sme\CC$-submodule of $\E$.
Then $M_{t}$ is a closed $B_{\rho(t)}\sme C_{\sigma(t)}$-submodule of
$E_{t}$.  By the Rieffel Correspondence, $M_{t}$ is a
$L_{t}\sme R_{t}$-\ib\ for the ideals
$L_{t}=\overline{\lip \B<M_{t},E_{t}>}$ in $B_{\rho(t)}$ and
$R_{t}=\overline{\rip\CC<E_{t},M_{t}>}$ in $C_{\sigma(t)}$.
Furthermore, $L_{t}\cdot E_{t}=M_{t}=E_{t}\cdot R_{t}$.

\begin{lemma}\label{lem-ideals-inv}
  In the current set-up, the ideal $R_{t}$ depends only on $\sigma(t)$
  and the ideal $L_{t}$ depends only on $\rho(t)$.  Hereafter, we will
  denote them by $R_{\sigma(t)}$ and $L_{\rho(t)}$, respectively.
\end{lemma}
\begin{proof}
  If $\sigma(t')=\sigma(t)$, then $t'=h\cdot t$.  Since $E_{t}$ is a
  $B_{\rho(t)} \sme C_{\sigma(t)}$-\ib\ and since $R_{t}$ and
  $R_{h\cdot t}$ are both ideals in $C_{\sigma(t)}$, to see that
  $R_{t}=R_{h\cdot t}$ it will suffice---by the Rieffel
  correspondence---to see that
  $E_{t}\cdot R_{t}=E_{t}\cdot R_{h\cdot t}$.  Since $\M$ is full, we
  have
  \begin{equation}
    \label{eq:18}
    E_{t}\cdot R_{h\cdot t}
    = E_{t} \cdot \overline{\rip\CC<E_{h\cdot t}, M_{h\cdot t}>} 
    = \overline{E_{t}\cdot \rip \CC<\overline{B_{h}\cdot
        E_{t}},\overline{B_{h}\cdot M_{t}}>}. 
  \end{equation}
  Clearly,
  \begin{equation}
    \label{eq:32}
    \overline{E_{t}\cdot \rip\CC<B_{h}\cdot E_{t},B_{h}\cdot
      M_{t}>} \subset \overline{E_{t}\cdot \rip \CC<\overline{B_{h}\cdot
        E_{t}},\overline{B_{h}\cdot M_{t}}>}.
  \end{equation}
  On the other hand, consider
  \begin{equation}
    \label{eq:34}
    e\cdot \rip\CC<f,g>
  \end{equation}
  with $e\in E_{t}$, $f\in \overline{B_{h}\cdot E_{t}}$, and
  $g\in \overline{B_{h}\cdot M_{t}}$.  Then there are sequences
  $(f_{i})\subset B_{h}\cdot E_{t}$ and
  $(g_{i})\subset B_{h}\cdot M_{t}$ such that $f_{i}\to f$ and
  $g_{i}\to g$ in norm in $E_{h\cdot t}$ and hence in $\E$.  Therefore
  $e\cdot \rip\CC<f_{i},g_{i}>\to e\cdot \rip\CC<f,g>$ in $\E$.  Since
  the convergence takes place in $E_{t}$, the convergence is in norm.
  It follows that
  \begin{equation}
    \label{eq:35}
    E_{t}\cdot \rip\CC<\overline{B_{h}\cdot E_{t}},\overline{B_{h}\cdot
      M_{t}}> \subset \overline{E_{t}\cdot \rip\CC<B_{h}\cdot
      E_{t},B_{h}\cdot M_{t}>}.
  \end{equation}
  Therefore we have equality in \eqref{eq:32}, and
  \begin{align}
    E_{t}\cdot R_{h\cdot t}
    &= \overline{E_{t}\cdot \rip\CC<B_{h}\cdot E_{t},B_{h}\cdot
      M_{t}>}
      \label{eq:ques}\\
    \intertext{which, using
    (E2)(iv), is}
    & = \overline{\lip\B<E_{t},B_{h}\cdot E_{t}> \cdot B_{h} \cdot M_{t}} \\
    \intertext{which, using
    (E2)(ii) and (E2)(iii), is}
    &= \overline{\lip\B<E_{t},E_{t}> \cdot B_{h}^{*}B_{h}\cdot
      M_{t}} \\
    \intertext{which, since
    $\overline{B_{h}^{*}B_{h}}=B_{s(h)}=B_{\rho(t)}=\overline{
    \lip\B<E_{t},E_{t}>}$ and since
    $\overline{B_{s(h)}\cdot M_{t}}=M_{t}$, is}
    &=M_{t}=E_{t}\cdot R_{t}.
  \end{align}
  Thus $R_{t}=R_{h\cdot t}$ as required.

  The proof for $L_{t}$ is similar.
\end{proof}

If $\sigma(t)=r(k)$, then $\rip\CC<M_{t},E_{t\cdot k}>$ is the
subspace of $C_{k}$ spanned by inner products of elements in $M_{t}$
with elements of $E_{t\cdot k}$.  Then given $k\in K$, we let
$\bigoplus_{\sigma(t)=r(k)} \rip\CC<M_{t},E_{t\cdot k}>$ denote the
subspace of $C_{k}$ generated by the summands.  Then
\begin{align}
  \label{eq:17}
  \rip\CC<\M,\E>
  := \coprod_{k\in K} \bigoplus_{\sigma(t)=r(k)}
  \rip\CC<M_{t},E_{t\cdot k}> 
  = \coprod_{k\in K} \bigoplus_{\sigma(t)=r(k)}
  \rip\CC<M_{t},\overline{E_{t}\cdot C_{k}}>
\end{align}
where the closure takes place in the Banach space $E_{t\cdot k}$.
  
\begin{lemma}
  \label{lem-tech-closure} In the setting above, both
  $\rip\CC<M_{t},E_{t}\cdot C_{k}>$ and
  $\rip\CC<M_{t},\overline{E_{t}\cdot C_{k}}>$ are norm dense in
  $R_{r(k)}\cdot C_{k}$ where we have invoked
  Lemma~\ref{lem-ideals-inv} to realize that $R_{t}$ depends only on
  $r(k)=\rho(t)$.
\end{lemma}
\begin{proof}
  Clearly,
  \begin{equation}
    \label{eq:19}
    \rip\CC<M_{t},E_{t}\cdot C_{k}>=
    \rip\CC<M_{t},E_{t}>\cdot C_{k}\subset R_{r(k)}\cdot C_{k}.
  \end{equation}
  Moreover,
  \begin{equation}
    \label{eq:22}
    \overline{\rip\CC<M_{t},E_{t}>\cdot C_{k}} =\overline{R_{r(k)}\cdot
      C_{k}} =R_{r(k)}\cdot C_{k}.
  \end{equation}
  This implies the first assertion.

  For the second, we just need to see that
  $\rip\CC<M_{t},\overline{E_{t}\cdot C_{k}}> \subset
  \overline{\rip\CC<M_{t},E_{t}>\cdot C_{k}}$.  To this end, suppose
  that $(c_{i})$ is a sequence in $E_{t}\cdot C_{k}$ converging to $c$
  in $E_{t\cdot k}$.  Then for any $m\in M_{t}$, the sequence
  $\bigl(\rip\CC<m,c_{i}>\bigr)$ converges to $\rip\CC<m,c>$ in $\CC$
  since $\rip\CC<\cdot,\cdot>$ is continuous on $\E*_{\rho}\E$.  Since
  the convergence takes place in $C_{k}$, the convergence is in norm
  by Lemma~\ref{lem-rel-top}.  Since each
  $\rip\CC<m,c_{i}>\in \rip\CC<M_{t},E_{t}>\cdot C_{k}$, the result
  follows.
\end{proof}

Using Lemma~\ref{lem-tech-closure} and \eqref{eq:17}, we have
\begin{equation}
  \label{eq:21}
  \rip\CC<\M,\E>\subset \coprod_{k\in K}  R_{r(k)}\cdot C_{k}.
\end{equation}
Moreover $\rip\CC<\M,\E>\cap C_{k}$ is norm dense in
$R_{r(k)}\cdot C_{k}$.

\begin{lemma}
  \label{lem-j-bb} In the current set-up,
  $\JJM:=\coprod_{k\in K} R_{r(k)}\cdot C_{k}$ is a Banach subbundle
  of~$\CC$.
\end{lemma}
\begin{proof}
  As in the proof of Proposition~\ref{prop-fd-sub-bun}, the issue is
  to see that $p\colon \JJM\to K$ is open.  Let $U$ be a nonempty
  (relatively) open set in $\JJM$.  Given $k\in p(U)$ it will suffice
  to show that given a sequence $(k_{i})$ converging to $k$ in $K$,
  $(k_{i})$ is eventually in $U$.  If this fails, then after passing
  to a subsequence and relabeling, we can assume $k_{i}\notin U$ for
  all $i$.

  Since $(\JJM)_{k}$ is $R_{r(k)}\cdot C_{k}$, we can find
  $c\in C_{r}$ and $c'\in R_{r(k)}$ such that $c'c\in U$ and
  $p(c'c)=k$.  Let $t\in T$ be such that $\sigma(t)=r(k)$.  Since
  $R_{r(k)}$ is the closure of $\rip\CC<E_{t},M_{t}>$ in $C_{r(k)}$,
  there is a sequence $(c_{i}')\subset \rip\CC<E_{t},M_{t}>$
  converging to $c'$ in norm.  But then $c_{i}'c\to c'c$ in norm.  But
  then $(c_{i}'c)$ is eventually in $U$.  Therefore, we may as well
  assume that $c' = \sum_{j=1}^{n} \rip\CC<e_{j},m_{j}>$ with each
  $e_{j}\in E_{t}$ and each $m_{j}\in M_{t}$.

  Since Banach bundles have enough sections we can find
  $f\in \Gamma(K;\CC)$, $g_{j}\in \Gamma(T;\E)$, and
  $h_{j}\in \Gamma(T;\M)$ such that $f(k)=c$, $g_{j}(t)=e_{j}$, and
  $h_{j}(t)=m_{j}$.

  Since $r(k_{i})\to r(k)=\sigma(t)$, and since $\sigma$ is open, we
  can pass to a subsequence, relabel, and assume that there are
  $t_{i}\in T$ such that $t_{i}\to t$ and $\sigma(t_{i})=r(k_{i})$.
  Then
  \begin{equation}
    \label{eq:1new}
    d(t_{i}):=\sum_{j=1}^{n}\rip\CC<g(t_{i}),h(t_{i})>\in R_{\sigma(t_{i})}
    \quad\text{and} \quad d(t_{i})f(k_{i}) \in R_{\sigma(t_{i})}\cdot C_{k_{i}}
    = (\JJM)_{k_{i}}.  
  \end{equation}
  Furthermore, $d(t_{i})f(k_{i})\to c'c$ in $\JJM$.  Hence
  $\bigl(d(t_{i})f(k_{i})\bigr)$ is eventually in $U$.  Since $p$ is
  continuous, $p\bigl(d(t_{i})f(k_{i})\bigr)=k_{i}$ is eventually in
  $p(U)$ which contradicts our assumptions on $(k_{i})$ and completes
  the proof.
\end{proof}

\begin{prop}
  \label{prop-j-ideal} In the current set-up,
  $\JJM:=\coprod_{k\in K} R_{r(k)}\cdot C_{k}$ is an ideal in $\CC$.
\end{prop}
\begin{proof}
  In view of Proposition~\ref{prop-weak-ideal}, we just have to show
  that $\JJM$ is a weak ideal.  For convenience, let
  $J_{k}=R_{r(k)}\cdot C_{k}$.
  
  Suppose that $(c,m)\in C_{l}\times J_{k}$ with $s(l)=r(k)$.  By
  Lemma~\ref{lem-tech-closure}, there is a sequence $(m_{i})$ in
  $\rip\CC<M_{t},E_{t}\cdot C_{k}>$ converging to $m$ in norm (and
  hence in $\CC$).  But then $c m_{i}\to c m$ in $\CC\cap C_{lk}$ in
  $\CC$ and hence in norm.  Since
  $M_{t}c^{*}\subset M_{t\cdot l^{-1}}$,
  $c m_{i}\in \rip\CC<M_{t\cdot l^{-1}},E_{t\cdot k}> \subset J_{lk}$
  (using \eqref{eq:17} and \eqref{eq:21}).  Since $J_{kl}$ is closed
  in norm in $C_{kl}$, it follows that $cm\in \JJM$.

  A similar argument shows that $mc\in \JJM$ if
  $(m,c)\in J_{k}\times C_{l}$ with $s(k)=r(l)$.  Hence $\JJM$ is a
  weak ideal as claimed.
\end{proof}

\begin{prop}
  \label{prop-bijection} We retain the current set-up.  Let $\JJ$ be
  an ideal in $\CC$.  Then $\M=\E\cdot \JJ$ is a Banach
  $\B\sme\JJ$-submodule of $\E$ and $\JJM=\JJ$.  Hence
  $\JJ\mapsto \E\cdot \JJ$ is a lattice isomorphism of the collection
  of ideals in $\CC$ to the collection of closed $\B\sme\CC$-submodules
  of $\E$.
\end{prop}

\begin{proof}
  By Proposition~\ref{prop-e-j}, $\M:=\E\cdot\JJ$ is a full Banach
  $\B\sme\JJ$-submodule, and
  $M_{t}=(\E\cdot \JJ)_{t}=E_{t}\cdot J_{\sigma(t)}$.  Thus, applying
  the Rieffel correspondence to $E_{t}$,
  \begin{equation}
    \label{eq:23}
    R_{\sigma(t)}=\overline{\text{span}}{\rip\CC<E_{t},M_{t}>}= 
    \overline{\text{span}}{\rip\CC<E_{t},E_{t}\cdot
    J_{\sigma(t)}>} =J_{\sigma(t)}.
  \end{equation}
  Thus in \eqref{eq:21}, $R_{r(k)}=J_{r(k)}$.  Therefore,
  \begin{equation}
    \label{eq:24}
    \JJM=\coprod_{k\in K} R_{r(k)}\cdot C_{k}=\coprod_{k\in K}
    J_{r(k)}\cdot C_{k}=\JJ,
  \end{equation}
  where the last equality comes from Equation~\eqref{eq:9a} of
  Lemma~\ref{lem-ib-cohen}.

  Now suppose that $\M$ is a full closed $\B\sme\CC$-submodule.  Then
  Proposition~\ref{prop-j-ideal} implies that $\JJM$ is an ideal in
  $\CC$ with $(\JJM)_{k}=R_{r(k)}\cdot C_{k}$.  Let
  $\M':=\E\cdot \JJM$.  In particular, if $u\in\ko$,
  $(\JJM)_{u}=R_{u}\cdot C_{u}=R_{u}$.  Thus
  $(\JJM)_{k}=R_{r(k)}\cdot C_{k} = C_{k}\cdot R_{s(k)}$ by
  Lemma~\ref{lem-ib-cohen}.  Then
  $\M'_{t}=(\E\cdot \JJM)_{t} =E_{t}\cdot R_{\sigma(t)}$.  This means
  \begin{align}
    \label{eq:2new}
    R'_{\sigma(t)}
    &=\overline{\rip\CC<E_{t},\M'_{t}>} \\
    &= \overline{\rip\CC<E_{t},E_{t}>\cdot R_{\sigma(t)}} \\
    &= \overline{\rip\CC<E_{t},E_{t}>\cdot R_{\sigma(t)}} \\
    &= \overline{C_{\sigma(t)}\cdot R_{\sigma(t)}}=R_{\sigma(t)}.
  \end{align}
  Therefore
  $\ M'_{t}=E_{t}\cdot R_{\sigma(t)}=E_{t}\cdot R'_{\sigma(t)}
  =M_{t}$.  Therefore $\E\cdot \JJM=\M$.
\end{proof}

By symmetry, we have a lattice isomorphism $\KK\mapsto \KK\cdot \E$
between the ideals in $\B$ and the full Banach $\B\sme\CC$-submodules
of $\E$.  Then we have the following Rieffel Correspondence for Fell
bundle equivalence.

\begin{thm}
  \label{thm-rieffel-corr-fell} Suppose that $\qe\colon \E\to T$ is a
  Fell-bundle equivalence between $\pb\colon \B\to H$ and
  $\pc\colon \CC\to K$.  Then there are lattice isomorphisms among the
  ideals of $\B$, the \emph{full} Banach $\B\sme\CC$-submodules of
  $\E$, and the ideals of \/$\CC$.  The correspondences are given as
  follows.
  \begin{enumerate}
  \item If $\JJ$ is an ideal in $\CC$, then the corresponding full
    Banach $\B\sme\CC$-submodule is
    \begin{equation}
      \label{eq:26}
      \E\cdot\JJ=\bigcup_{t\in T}E_{t}\cdot J_{\sigma(t)}.
    \end{equation}
  \item If $\M$ is a full Banach $\B\sme\CC$-submodule, then for each
    $t\in T$, $M_{t}$ is a $L_{\rho(t)}\sme R_{\sigma(t)}$-\ib\ for
    ideals $L_{\rho(t)} = \overline{\lip\B<M_{t},E_{t}>}$ in
    $B_{\rho(t)}$ and $R_{\sigma(t)}= \overline{\rip\CC<E_{t},M_{t}>}$
    in $\CC_{\sigma(t)}$.  Then the corresponding ideals $\JJM$ in
    $\CC$ and $\KKM$ in $\B$ are given by
    \begin{align}
      \label{eq:27}
      \JJM
      &=\bigcup_{k\in K} R_{r(k)}\cdot C_{k} \quad\text{and} \quad
        \KKM= \bigcup_{h\in H} L_{r(h)} \cdot B_{h}.
    \end{align}
  \item If $\KK$ is an ideal in $\B$, then the corresponding
    $\B\sme\CC$-submodule is
    \begin{equation}
      \label{eq:28}
      \KK\cdot \E= \bigcup_{t\in T} K_{\rho(t)}\cdot E_{t}.
    \end{equation}
  \end{enumerate}
\end{thm}

The following is a generalization of
\cite{rw:morita}*{Proposition~3.24}.

\begin{cor}
  \label{cor-rie-corr} Suppose that $\qe\colon \E\to T$ is a Fell
  bundle equivalence between $\pb\colon \B\to H$ and
  $\pc\colon \CC\to K$.  If $\JJ$ is an ideal in $\CC$, then the
  corresponding ideal $\KK$ in $\B$ is
  $\coprod_{h\in H} H_{r(h)}\cdot B_{h}$ where
  \begin{equation}
    \label{eq:4new}
    H_{\rho(t)}=\overline{\lip\B<E_{t}\cdot J_{\sigma(t)},E_{t}>}.
  \end{equation}
\end{cor}

\begin{proof}
  If $\JJ$ is an ideal in $\CC$, then according to
  Theorem~\ref{thm-rieffel-corr-fell}(a), the corresponding full
  Banach $\B\sme \CC$-submodule is $\M=\E\cdot \JJ$.  Then using
  Theorem~\ref{thm-rieffel-corr-fell}(b), the corresponding ideal
  $\KK$ in $\B$ is $\bigcup_{h\in H} L_{r(h)}\cdot B_{h}$ where
  $L_{r(h)}$ is given by the right-hand side of \eqref{eq:4new} for
  any $t\in T$ such that $\rho(t)=r(h)$.  This gives the result.
\end{proof}

\section{Extending the Rieffel Correspondence}
\label{sec:extend-rieff-corr}

Now we want to state and prove the analogues for Fell bundles of parts (c)~and (d) of
Theorem~\ref{thm-rief-corr}.

\begin{prop}
  \label{prop-me} Suppose that $\qe\colon \E\to T$ is a Fell-bundle
  equivalence between $\pb\colon \B\to H$ and $\pc\colon \CC\to K$.  Suppose that
  $\JJ$ is an ideal in $\CC$ and that $\M$ and $\KK$ are the
  corresponding full Banach $\B\sme \CC$-submodule in $\E$ and ideal in $\B$.
  Then $\M$ is a Fell-bundle  equivalence between $\KK$ and $\JJ$.
\end{prop}

\begin{proof}
  Since $\M$ is a $\B\sme\CC$-submodule of $\E$, we clearly have a left
  $\KK$-action and a right $\JJ$-action satisfying
  (E1).

  For (E2), we claim that it suffices to
  let $\lip\KK<e,f>=\lip\B<e,f>$ and $\rip\JJ<e,f>=\rip\CC<e,f>$.  To
  see this, note that if $(e,f)\in \M*_{\sigma}\M$, then we can assume
  $(e,f) \in M_{t}\times M_{h^{-1}\cdot t}$ for some $h\in H$ and $t\in T$.   But
  $M_{t}=K_{\rho(t)}\cdot E_{t}$.  However, $\lip\B<K_{\rho(t)}\cdot
  E_{t},M_{h^{-1}\cdot t}>=K_{\rho(t)}\cdot
  \lip\B<E_{t},M_{h^{-1}\cdot t}>\subset K_{\rho(t)}\cdot
  B_{h}=K_{h}$.   Therefore $\lip\KK<\cdot,\cdot>$ is $\KK$-valued.
  Similarly, $\rip\JJ<\cdot,\cdot>$ is $\JJ$-valued.  The rest of
  (E2) follows from the given
  properties of $\lip\B<\cdot,\cdot>$ and $\rip\CC<\cdot,\cdot>$.

  For (E3), the fact that $M_{t}$ is a
  $K_{\rho(t)}\sme J_{\sigma(t)}$-\ib\ follows from the Rieffel Correspondence
  (part~(c) of Theorem~\ref{thm-rief-corr}).
\end{proof}

\begin{prop}
  \label{prop-equiv-quotient} Let $\qe\colon \E\to T$ be an equivalence
  between $\pb\colon \B\to H$ and $\pc\colon \CC\to K$.  Suppose that
  $\JJ$ is an 
  ideal in $\CC$ and that $\M$ and $\KK$ are the corresponding full
  Banach $\B\sme\CC$-submodule in $\E$ and ideal in $\B$,
  respectively.  Then the quotient Banach bundle $\E/\M$ is an
  equivalence between $\B/\KK$ and $\CC/\JJ$.
\end{prop}

\begin{proof}
  We let $\qk\colon \B\to \B/\KK$ and $\qj\colon \CC\to \CC/\JJ$ be the quotient
  maps.  Then 
  the given left and right actions of $\B$ and $\CC$ on $\E$ induce
  left and right actions of $\B/\KK$ and $\CC/\JJ$ on $\E/\M$ in the
  expected way:
  \begin{equation}
    \label{eq:40}
    \qk(b)\cdot q(e)=q(b\cdot e)\quad\text{and} \quad q(e)\cdot
    \qj(c)=q(e\cdot c)
  \end{equation}
  assuming that $b\cdot e$ and $e\cdot c$ are defined.

  To see that these actions are continuous, we use the fact that $q$,
  $\qk$, and $\qj$ are open as well as continuous
  (Proposition~\ref{prop-ban-bund-quotient}).   Suppose that
  $q(e_{i})\to q(e)$ while $\qk(b_{i})\to\qk(b)$ with $b_{i}\cdot
  e_{i}$ defined for all $i$.   We need to verify that $q(b_{i}\cdot
  e_{i})\to q(b\cdot e)$.   For this, it suffices to see that every
  subnet has a subnet converging to $q(b\cdot e)$.  But after passing
  to a subnet and relabeling, the openness of the quotient maps means
  we can pass to another subnet and assume that $e_{i}'\to e$ and
  $b_{i}'\to b$ with $q(e_{i}')=q(e)$ and $\qk(b_{i}')=\qk(b_{i})$.
  Then the continuity of the quotient maps implies that $q(b_{i}\cdot
  e_{i}) = q(b_{i}'\cdot e_{i}')\to q(b\cdot e)$ as required.

 We also have
 \begin{align}
   \label{eq:41}
   \|q(b\cdot e)\|
   &\le \inf\set{\|b'\cdot e'\|:\text{$\qk(b')=\qk(b)$ and
     $q(e')=q(e)$}} \\
   &\le \inf\set{\|b'\|\|e'\|:\text{$\qk(b')=\qk(b)$ and
     $q(e')=q(e)$}} \\
   &= \|\qk(b)\|\|q(e)\|.
 \end{align}

 Therefore $\B/\KK$ acts on the left of $\E/\M$.  
 The argument for the right action is similar.

 Now we need to verify the axioms in Definition~\ref{def-equi}.
 (E1) is immediate since
 $\E$ is an equivalence.  For Axiom~(E2), we define
 \begin{equation}
   \label{eq:42}
   \rip\CC/\JJ<q(e),q(f)>:= \qj\bigl(\rip\CC<e,f>\bigr) \quad
   \text{and}
   \lip\B/\KK<q(e),q(f)>= \qk\bigl(\lip\B<e,f>\bigr).
 \end{equation}
 It is not hard to check that these pairings are
 well-defined.
 Then properties (i), (ii), (iii), and (iv) follow from the
 corresponding properties for $\E$ and the observation that the
 quotient maps are multiplicative.  The continuity follows using the
 continuity and openness of the quotient maps as we did above for the
 left and right actions.

 Of course, Axiom~(E3) is clear.
\end{proof}

\section{At the \cs-Level}
\label{sec:at-cs-level}

Since the previous exposition did not require it, we have purposely
avoided discussing the Fell-bundle \cs-algebras that are associated to
a Fell bundle.   However, there is an obvious question: how is our
Rieffel correspondence for ideals in equivalent Fell bundles related to the
standard Rieffel correspondence for ideals in Morita equivalent
\cs-algebras?   In order that there be \cs-algebras,
we now have to assume our groupoids
have Haar systems.   In order to apply the Equivalence Theorem---that
is, \cite{muhwil:dm08}*{Theorem~6.4}---we also need our Fell bundles 
to be separable.   

We return to the set-up
in Section~\ref{sec:rieff-corr-fell}: we let $\qe\colon\E\to T$ be an
equivalence between the separable Fell bundles $\pb\colon \B \to H$
and $\pc\colon\CC\to K$.  In particular, $T$ is a
$(H,K)$-equivalence\footnote{Although it is required, we note that $T$
  must be second countable since $H$ and $K$ are
  \cite{wil:toolkit}*{Proposition~2.53}.} and we let
$\rho\colon T\to \ho$ and $\sigma\colon T\to \ko$ be the open moment
maps.

Then the Equivalence
Theorem implies that $\cs(\ho;\B)$ and $\cs(\ko;\CC)$ are Morita
equivalent via an imprimitivity bimodule $\X$ which is the completion
of $\X_{0}:=\Gamma_{c}(T,\E)$ with the actions and inner products
given in \cite{muhwil:dm08}*{Theorem~6.4}.  Then we can let
\begin{equation}
  \label{eq:44}
  \xind:\I\bigl(\cs(\ko;\CC)\bigr)\to \I\bigl(\cs(\ho;\B)\bigr)
\end{equation}
be the
classical Rieffel lattice isomorphism.

If $\JJ$ is an ideal in $\CC$, then as shown in
\cite{ionwil:hjm11}*{Lemma~3.5}, the identity map $\iota$ induces an
isomorphism of $\cs(\ko;\JJ)$ onto the ideal $\Ex(\JJ)$ which is the
closure of $\iota(\Gamma_{c}(\ko;\CC))$ in $\cs(\ko;\CC)$.

Let $\JJ$ be an ideal in $\CC$ and $\KK$ the corresponding ideal in
$\B$ as in Theorem~\ref{thm-rieffel-corr-fell}.    The goal here is to
establish that the two Rieffel correspondences are compatible in that
\begin{equation}
  \label{eq:16}
  \xind\bigl(\Ex(\JJ)\bigr)=\Ex(\KK).
\end{equation}

By \cite{rw:morita}*{Proposition~3.24}, the left hand side of
\eqref{eq:16} is
\begin{equation}
  \label{eq:25}
  \ospan\set{\Lip <x\cdot b,y>:\text{$x,y\in \X$ and $b\in
      \Ex(\JJ)$}}
\end{equation}
where $\lip<\cdot,\cdot>$ is the $\Gamma_{c}(\ho;\B)$-valued inner
product from Equation~(6.3) of \cite{muhwil:dm08}*{Theorem~6.4}. In
particular, if $x,y\in \X_{0}$ and $b\in \Gamma_{c}(\ko;\JJ)$, then
provided $\rho(t)=s(h)$,
\begin{align}
  \label{eq:30}
  \Lip<x\cdot b,y>(h)
  &= \int_{K}\lip\B<x\cdot b (h \cdot t\cdot k),y(t\cdot
    k)>\,d\lambda_{K}^{\sigma(t)} (k)
\end{align}
where according to Equation (6.2) of \cite{muhwil:dm08}*{Theorem~6.4}
we have
\begin{align}
  \label{eq:31}
  x\cdot b(h\cdot t \cdot k)=\int_{K}x(h\cdot t \cdot k l)b(l^{-1})
  \,d\lambda_{K}^{s(k)} (l).
\end{align}

Let $\M=\E\cdot \JJ=\KK\cdot \E$.  Note that $\M$ is a $\KK\sme \JJ$-equivalence.
Then the integrand in \eqref{eq:31} is in the Banach space
$M_{h\cdot t\cdot k}$ for all $l$.  Hence
\begin{equation}
  \label{eq:33}
  x\cdot b(h\cdot t \cdot k)\in M_{h\cdot t\cdot k}= K_{r(h)} \cdot
  E_{h\cdot t\cdot k}. 
\end{equation}
Plugging into \eqref{eq:30}, and using (E2)(iii) of
Definition~\ref{def-equi}, we clearly have
\begin{equation}
  \label{eq:36}
  \Lip<x\cdot b,y>(h)\in K_{r(h)}\cdot B_{h}=K_{h}.
\end{equation}
It follows that
\begin{equation}
  \label{eq:38}
  \xind\bigl(\Ex(\JJ)\bigr) \subset \Ex(\KK).
\end{equation}

But we can also work with $\xind^{-1}$.   Then
\begin{equation}
  \label{eq:39}
  \xind^{-1}(\Ex(\KK))=\ospan\set{\Rip<x,c\cdot y>:\text{$x,y\in \X$
      and $c\in\Ex(\KK)$.}}
\end{equation}
Then a similar argument to the above shows that
\begin{equation}
  \label{eq:43}
  \xind^{-1}(\Ex(\KK))\subset \Ex(\JJ).
\end{equation}
Now \eqref{eq:16} follows by applying $\xind$ to both sides of
\eqref{eq:43}.


\def\noopsort#1{}\def\cprime{$'$} \def\sp{^}

\end{document}